\documentclass[12pt]{article}

\usepackage{amsfonts}

\def\be{\begin{equation}}
\def\ee{\end{equation}}
\def\l{\langle}
\def\r{\rangle}
\def\p{\parallel}
\def\R{I\!\!R}
\def\N{I\!\!N}

\def\bea{\begin{eqnarray}}
\def\eea{\end{eqnarray}}

\def\cR{{\cal R}}
\def\cN{{\cal N}}

\def\12{\frac{1}{2}}

\newtheorem{remark}{Remark}

\newtheorem{definition}{Definition}
\newtheorem{proposition}{Proposition}

\newenvironment{proof}{\noindent\textbf{Proof.}
  }{\hspace*{\fill}$\spadesuit$ \\[2mm]}

\begin{document}
\begin{center}
{\Large \bf A note on Kaczmarz algorithm  with remotest set control sequence}

\vspace*{1cm}
{\small CONSTANTIN POPA}

\vspace*{1cm}
Ovidius University of Constanta, Blvd. Mamaia 124, Constanta 900527, Romania; {\tt cpopa@univ-ovidius.ro}

\vspace*{0.2cm}
              ``Gheorghe Mihoc - Caius Iacob'' Institute of Statistical Mathematics and Applied Mathematics of the Romanian Academy, Calea 13 Septembrie, Nr. 13, Bucharest 050711, Romania\\
\end{center}

{\bf Abstract.} 
In this paper we analyse the Kaczmarz projection algorithm with remotest set control of projection indices. According to this procedure, in each iteration the projection index is one which gives the maximal absolute value of the corresponding residual. We prove that for underdetrmined full row rank systems and under some assumptions valid for problems arising in algebraic reconstruction of images in computerized tomography, this selection procedure has the property that each row index is selected at least once during the Kaczmarz algorithm iterations.



\section{Introduction}
\label{intro}

For $A$  an $m \times n$ (real) matrix $A$ and  $b \in \R^m$ in this paper we will consider the (consistent) system of linear equations 
\begin{equation}
\label{1}
 Ax = b,
\end{equation}
and denote by  $S(A; b)$ the set of its solutions and by $x_{LS}$ the minimal (Euclidean)  norm one. We  will use the notations $A^T, A_i, A^j, \cR(A), \cN(A)$, ${rank}(A)$, and $P_V$  for the transpose, $i$-th row, $j$-th column, range and null space of $A$, the rank of A, and the projection onto a nonempty closed convex set  $V$.  We know that
\be
\label{1-1}
\cR(A^T) = {sp}{\{A_1, A_2, \dots, A_m\}}, ~x_{LS} \in \cR(A^T).
\ee
Also $\left\langle \cdot, \cdot \right\rangle$ and $\parallel \cdot \parallel$ will denote the Euclidean scalar product and norm and all the vectors appearing in the paper will be considered as column vectors. If $H_i = \left\{x \in \R^n, \left\langle x, A_i \right\rangle = b_i \right\}$ is the hyperplane determined by the $i$-th equation of the  system (\ref{1}) we have
\begin{equation}\label{res:PHi}
P_{H_i}(x) = x - \frac{\left\langle x, A_i \right\rangle - b_i}{\left\Vert A_i \right\Vert ^2}A_i.
\end{equation}
 The Kaczmarz algorithm with single projection (for short Kaczmarz) is the following.\\
{\bf Algorithm Kaczmarz}\\
{\it Initialization:} $x^0 \in \R^n$ \\
{\it Iterative step:} for $k = 0, 1, \dots$ select $i_k \in \{ 1, 2, \dots, m \}$ and compute $x^{k+1}$ as 
        \be
        \label{ksp}
        x^{k+1}= x^{k} -  \frac{\l x^{k}, A_{i_k} \r -  b_{i_k}}{\p A_{i_k} \p^2} A_{i_k}.
        \ee
 For an almost complete overview on the selection procedures in Kaczmarz algorithm  see \cite{ycrow}, \cite{censta} (section 5.1), \cite{comb}, \cite{ccp11} and references therein. But, an important problem when considering a selection procedure seems to be the following: {\it ``sooner or latter'' during the iterations each (row) projection index $i_k$ must appear}. This was clearly formulated in \cite{ccp11} as follows ($\N$ will denote the set of natural numbers $\{ 0, 1, 2, \dots, \}$).
\begin{definition}
\label{def1}
Given a monotonically increasing sequence $\{\tau_{k}\}_{k=0}^{\infty} \subset \N$, a mapping $i: \N \rightarrow \{1, 2, \dots, m\}$ is called a {\em control with respect to the sequence $\{\tau_{k}\}_{k=0}^{\infty}$} if it defines a {\em  sequence $\{i(t)\}_{t=0}^{\infty}$}, such that for all $k \geq 0$,
\begin{equation}\label{def:inclusion}
\{1, 2, \dots, m\} \subseteq \{i(\tau_k), i(\tau_k + 1), \dots, i(\tau_{k + 1} - 1)\}.
\end{equation}
The set $\tau_k, \tau_k + 1, \dots, \tau_{k+1} - 1$ is called the $k$-th {\em window} (with respect to the given sequence $\{\tau_{k}\}_{k=0}^{\infty}$) and  $C_k = \tau_{k + 1} - \tau_k$ its {\em lenght}. If the sequence of lenghts  $(C_k)_{k \geq 0}$ is bounded the control $\{i(t)\}_{t=0}^{\infty}$ itself is called {\em bounded}.  If the sequence of lenghts  $(C_k)_{k \geq 0}$ is unbounded the control $\{i(t)\}_{t=0}^{\infty}$ itself is called an {\em expanding control}. 
\end{definition}
In the same paper \cite{ccp11} there are defined different types of bounded and expanding control sequences. But, there are also other types of control sequences which are not included in the above definition. Two well-known such examples are the {\it random control} and {\it remotest set control} (called in the present paper {\em Maximal Residual control} (MR, for short). 
\begin{itemize}
\item {\bf Maximal Residual control:} Select $i_{k} \in \{ 1, 2, \dots, m \}$ such that
\be\label{MR}
|\l A_{i_k}, x^{k-1} \r - b_{i_k}| = \max_{1 \leq i \leq m} |\l A_{i}, x^{k-1} \r - b_{i}|.
\ee

\item {\bf Random control:} Let the set $\Delta_m \subset \R^m$ be defined by  
\be
\label{0010}
\Delta_m = \{ x \in \R^m, x \geq 0, \sum_{i=1}^m = 1 \},
\ee
define the discrete probability distribution 
\begin{equation}\label{RK}
p \in \Delta_{m},\ p_{i} = \frac{\|A_{i}\|^{2}}{\|A\|^{2}_{F}}, 
\ i = 1, \dots, m,
\end{equation}
and select $i_k \in \{ 1, 2, \dots, m \}$
\be\label{RK1}
i_{k} \sim p .
\ee 
\end{itemize}
At least related to author's knowledge, there are no results saying that the above two control sequences satisfy the previously mentioned property, i.e.  ``sooner or latter'' during the iterations of Kaczmarz algorithm (\ref{ksp}) with that specific choice of the control sequence, each (row) projection index $i_k$ must appear. More clear, we formulate this property as follows: ``
Determine appropriate assumptions on (\ref{1}) such that 
\begin{equation}
\label{cond}
\forall i\in \{1, 2, \dots, m\}, ~\exists ~k ~\geq 0,~{\rm with}~ ~i_k = i.
\end{equation}
In the rest of the paper we will analyse this property for the Kaczmarz algorithm with Maximal Residual control sequence (MRK, for short) and show that it  exists a case in which the property (\ref{cond}) can be theoretically proved.
\section{Algorithm MRK}
\label{secMR}

We consider in this section Kaczmarz algorithm (\ref{ksp}) in which the Maximal Residual  control procedure is used for selecting the projection indices in each iteration (caled MRK algorithm).
\newpage
{\bf Algorithm MRK}\\

{\it Initialization.}  $x^0 \in \R^n$; \\

{\it Iterative step.} Select $i_k \in \{1, \dots, m \}$ such that
\be
\label{2}
|r_{i_k}| =  \max_{1\leq i\leq m}{|r_i|}, ~{\rm where}~ r=Ax^k-b \in \R^m,
\ee
and perform the projection
\begin{equation}
\label{3}
   x^{k+1}  = P_{H_{i_k}}(x^k), \forall k \geq 0. 
\end{equation}
The following result gives us a sufficient condition such that the property (\ref{cond}) holds.
\begin{proposition}
\label{prop1}
Let $m \leq n$ and suppose that
\begin{equation}
\label{4}
{rank}{A} = m,
\end{equation}
and
\be
\label{3-3p}
P_{\cR(A^T)}(x^0)-x_{LS} = \sum_{i = 1}^m \gamma_i A_i.
\ee
If
\be
\label{3-3pp}
\gamma_i > 0, \forall i=1, \dots, m,
\ee
then (\ref{cond}) is true for the MRK algorithm.
\end{proposition}
\begin{proof} 
Suppose that (\ref{cond}) is not satified and let $i_0 \in \{1, 2, \dots, m\}$ be such that $i_k \neq i_0$, for all $k \geq 0$. Then, (\ref{res:PHi}) yields that 
\begin{equation}
\label{eq0}
x^k = x^0 + {\displaystyle  \sum_{1 \leq i \leq m; i \neq i_0}{\alpha^k_i A_i}}, 
\end{equation}
with $\alpha^k_i \in \R$, hence 
\be
\label{eq1}
x^k \in x^0 + {\rm span}(A_i, i=1, \dots, m, i \neq i_0).
\ee
In \cite{a84}) the author proved that for consistent systems as (\ref{1}) (which holds in our case because of the assumption (\ref{4})) the sequence $(x^k)_{k \geq 0}$ generated with the MRK algorithm converges and 
\begin{equation}
\label{eq2}
\displaystyle \lim_{k \rightarrow \infty} x^k = P_{{\cal{N}}(A)}(x^0) + x_{LS}.
\end{equation}
Since the set $x^0 + {\rm span}(A_i, i=1, \dots, m, i \neq i_0)$ is closed, from (\ref{eq1})) it results that the limit vector in (\ref{eq2})) belongs to the same set, thus
$$
P_{{\cal{N}}(A)}(x^0) + x_{LS} - x^0 = x_{LS} - 
P_{\cR(A^T)}(x^0) \in {\rm span}(A_i, i=1, \dots, m, i \neq i_0).
$$
 This contradicts the hypothesis  (\ref{3-3pp}) and completes the proof.
\end{proof}
The above result tells us that, in the hypothesis (\ref{3-3pp}) the remotest set control is {\it a kind of  expanding} control (according to \cite{ccp11}). Regarding the possibility to fulfil this hypothesis we give the following result.
\begin{proposition}
\label{pnew}
Let 
\be
\label{3-0}
x^0 = \sum_{i=1}^m \beta_i A_i \in \cR(A^T)
\ee
 and suppose that 
\be
\label{3-1}
A_i \neq 0, ~A_{ij} \geq 0, \forall i, j ~{\rm and}~ \p x_{LS} \p ~\leq~ M,
\ee
for some $M \geq 0$. If the scalars $\beta_i$ satisfy
\be
\label{3-2}
\beta_i > \frac{M}{M_i}, ~M_i=\max_{1 \leq j \leq n}A_{ij} > 0, ~\forall i=1, \dots, m,
\ee
then 
\be
\label{3-3}
P_{\cR(A^T)}(x^0)-x_{LS} = x^0-x_{LS}= \sum_{i = 1}^m \gamma_i A_i, ~{\rm with}~ \gamma_i > 0, \forall i=1, \dots, m.
\ee
\end{proposition}
\begin{proof}
Let
\be
\label{3-4}
x_{LS}=(x_1, \dots, x_n)^T=\sum_{i = 1}^m \alpha_i A_i.
\ee
We distinguish the following two cases.\\
{\bf Case 1.} Let $i_0 \in \{ 1, \dots, m \}$ be an index such that in (\ref{3-4}) $\alpha_{i_0} \leq 0$. Then, if we take $\beta_{i_0} > 0$ for the corresponding $\gamma_{i_0}$ in (\ref{3-3}) we obtain $\gamma_{i_0} > 0$, which fits into our conclusion.\\
{\bf Case 2.} According to {\bf Case 1} we may suppose that in (\ref{3-4}) we have
\be
\label{3-5}
\alpha_i > 0, \forall i =1, \dots, m.
\ee
From (\ref{3-4})  we get
$$
0 \leq x_j=\sum_{i=1}^m \alpha_i A_{ij}, \forall j=1, \dots, n,
$$
which gives us
$$
x_j=| x_j | \leq \sqrt{\sum_{q=1}^n x^2_q} = \p x_{LS} \p \leq M, \forall j=1, \dots, n,
$$
and therefore
\be
\label{3-6}
0 \leq \sum_{i=1}^m \alpha_i A_{ij} \leq M, \forall j=1, \dots, n.
\ee
If $i \in \{ 1, \dots, m \}$ is arbitrary fixed, from (\ref{3-6}) we obtain
\be
\label{3-7}
0 \leq  \alpha_{i} A_{ij} \leq M, \forall j=1, \dots, n.
\ee
Again because of our assumptions (\ref{3-1}) it results that it exists at least one index $j$ such that $A_{ij} > 0$, which tell us that
\be
\label{3-8}
M_i = \max_{1 \leq j \leq n} A_{ij} > 0.
\ee
From (\ref{3-6}) - (\ref{3-7}) we obtain that the coefficients $\alpha_i$ from (\ref{3-4}) should satisfy
\be
\label{3-9}
\alpha_i \leq \frac{M}{M_i}, \forall i=1, \dots, m
\ee
with $M_i$ defined in (\ref{3-8}). Hence, in order to get the conclusion (\ref{3-3}) we must take $\beta_i$ as in  (\ref{3-2}) and the proof is complete.
\end{proof}
\begin{remark}
\label{rem3-10}
If $A$ is a scanning matrix from ART in CT, the second assumption in (\ref{3-1}) is satisfied.  The first assumption is usually imposed for projection-based iterative methods. Anyway, it is not a restrictive condition because any zero row from $A$ can be eliminated from the begining without changing the solution set of (\ref{1}). The third assumption is also connected with the ART; indeed we usually have information about the components of the solutions $z=(z_1, \dots, z_n)^T  \in S(A; b)$ of the form $0 \leq z_j \leq C, \forall j=1, \dots, n$.  This gives us
$$
\p x_{LS} \p \leq \p z \p = \sqrt{\sum_{j=1}^n z^2_j} \leq \sqrt{n} C.
$$
\end{remark}

\end{document}